\documentclass[leqno]{amsart}
\usepackage{amsmath}
\usepackage{mathtools}
\usepackage{amssymb}
\usepackage{amsthm}
\usepackage{graphicx}
\usepackage{enumerate}
\usepackage[mathscr]{eucal}
\usepackage{upgreek}
\theoremstyle{plain}
\newtheorem{theorem}{Theorem}[section]

\newtheorem{cor}{Corollary}[theorem]

\theoremstyle{definition}
\newtheorem{definition}{Definition}[section]
\newtheorem{remark}{Remark}[section]

\usepackage[pagewise]{lineno}

\begin{document}

         \title[Orthogonality in a vector space with a topology ]{Orthogonality in a vector space with a topology And a generalization of Bhatia-$\mathbf{\breve{S}}$emrl Theorem}
\author[Sain, Roy, Paul]{Debmalya Sain, Saikat Roy, Kallol Paul}
                \newcommand{\acr}{\newline\indent}

\address[Sain]{Department of Mathematics\\ Indian Institute of Science\\ Bengaluru 560012\\ Karnataka \\INDIA\\ }
\email{saindebmalya@gmail.com}

\address[Roy]{Department of Mathematics\\ National Institute of Technology Durgapur\\ Durgapur 713209\\ West Bengal\\ INDIA}
\email{saikatroy.cu@gmail.com}

\address[Paul]{Department of Mathematics\\ Jadavpur University\\ Kolkata 700032\\ West Bengal\\ INDIA}
\email{kalloldada@gmail.com}

\thanks{The research of Dr. Debmalya Sain is sponsored by Dr. D. S. Kothari Post-doctoral fellowship. The research of Mr. Saikat Roy is supported by CSIR MHRD in terms of Junior research Fellowship under the supervision of Prof. Satya Bagchi.}


\begin{abstract}
We introduce the notion of orthogonality in a vector space with a topology on it. To serve our purpose, we define orthogonality space for a given vector space $ X $, using the topology on it. We show that for a suitable choice of orthogonality space, Birkhoff-James orthogonality in a Banach space is a particular case of the orthogonality introduced by us. We characterize the right additivity of orthogonality in our setting and obtain a necessary and sufficient condition for a Banach space to be smooth as a corollary to our characterization. Finally, using our notion of orthogonality, we obtain a topological generalization of the Bhatia-$\breve{S}$emrl Theorem.
\end{abstract}


\maketitle
\section{Introduction}

The purpose of the present work is to generalize the classical concept of orthogonality to the setting of vector space with a topology on it. The importance and the all-pervasiveness of orthogonality in Euclidean geometry can be hardly overemphasized. Roberts \cite{R}, Birkhoff \cite{B}, James \cite{J,Ja}, and Day \cite{D} were the first mathematicians to introduce and study orthogonality in the general setting of normed linear spaces and metric linear spaces. Although Birkhoff-James orthogonality is arguably the most natural orthogonality type in a normed linear space, it is now well-known that there are several distinct concepts of orthogonality in a normed linear space which are equivalent only if the norm is induced by an inner product. In particular, this illustrates that the various generalizations of the concept of orthogonality in Euclidean spaces to the setting of normed linear spaces are in some sense quite natural. In recent times, several mathematicians have explored the geometry of bounded linear operators between Banach spaces from the point of view of Birkhoff-James orthogonality \cite{BS, S, Sa, SP}. Indeed, it is well-accepted that Birkhoff-James orthogonality techniques are extremely valuable in the study of geometry of Banach spaces. In this article, we present a natural generalization of Birkhoff-James orthogonality in a vector space with a topology on it. As far as we understand, this is the minimal requirement on a space to have a satisfactory concept of orthogonality.

Let $ ( X, + , \cdot ) $ denote a real vector space of dimension strictly greater than one and let $ \uptau $ be a topology on $ X $. The scalar field $ \mathbb{R} $ is always endowed with the usual topology. It is worth mentioning that, in general the vector space operations need not be continuous under $ \uptau $. Whenever the vector space operations happen to be continuous under $ \uptau $ and every singleton set in $ (X, \uptau ) $ is closed, the pair $ ( X, \uptau ) $ is called a topological vector space. It should be noted that every Banach space is itself a topological vector space under the topology induced by the norm. A Banach space $ (X, \|\cdot \|) $, when considered as a topological vector space, will be denoted by $ (X, \uptau_{\|\cdot\|}), $ where $ \uptau_{\| \cdot \|} $ is the topology on $ X $ , induced by the norm. For the sake of brevity, throughout this article, we will use the same letter $ X $ to stand for either of $ (X, + , \cdot) $ and $ (X, \|\cdot \|) $. The meaning will be clear from the context. Given a Banach space $ X, $ let $ B_X =\{ x\in X : \| x \| \leq 1 \} $ and $ S_X = \{ x\in X : \| x \| =1 \} $ denote the unit ball and the unit sphere of $ X $ respectively. Let $ X^* $ denote the dual space of $ X. $ Given any two elements $ x, y \in X,  $ we say that $ x $ is Birkhoff-James orthogonal to $ y, $ written as $ x \perp_B y, $ if $ \| x \| \leq \| x + \lambda y \| $ for all scalars $ \lambda $. In Theorem $ 2.1 $ of \cite{Ja}, James proved that $ x \perp_B y $ if and only if there exists a continuous linear functional $ f \in X^* $ such that $ f(x)=\|f\|\|x\| $ and $ f(y)=0. $ This particular characterization of Birkhoff-James orthogonality serves as an excellent guide for introducing orthogonality in a vector space with a topology. We recall from \cite{Ja} that Birkhoff-James orthogonality is said to be right additive in $ X $ if given any $ x, y, z \in X, $ $ x \perp_B y $ and $ x \perp_B z $ implies that $ x \perp_B (y+z). $ \\

The first step towards formulating a meaningful definition of orthogonality in a vector space with a topology is to consider the projective relation $ `` \rho" $ on $ X\setminus \{ 0 \} $ defined in the following way:\\

 Given any $ u, v \in X\setminus \{ 0 \} $, $ u~ \rho ~ v $ if and only if $ u=tv $ for some non-zero $ t\in \mathbb{R} $.\\ 

\noindent It is easy to see that $ \rho  $ is an equivalence relation on $ X\setminus \{ 0 \} $. Therefore $ \rho  $ partitions $ X\setminus \{ 0 \} $ into disjoint equivalence classes $ [u] $, where $ [u] = \{ v: v=tu , ~ t\in \mathbb{R}\setminus \{0\} \} $. $ X\setminus \{ 0 \}, $ equipped with this equivalence relation, is denoted by $ ( X\setminus \{ 0 \} )/{\rho} $. We now present the following definition which also plays a crucial role in introducing orthogonality in a vector space with a topology. 

\noindent 
\begin{definition}\label{ admissible}
Let $ X $ be a vector space and let $ \rho $ be the projective equivalence relation on $ X\setminus \{ 0 \} $. A subset $ \mathcal{A}_X $ of $ X $ is said to be $ \rho $-admissible if $ \mathcal{A}_X $ contains exactly one element from each distinct equivalence class under $ \rho $.
\end{definition}

\noindent We would like to remark that  by applying the axiom of choice, we can choose exactly one element from each distinct equivalence class under $ \rho $. Therefore, the existence of at least one $ \rho $-admissible subset is always guaranteed. We further note that for each non-zero element $ x\in X $, $ x $ corresponds to a unique element $ a_x $ of $ \mathcal{A}_X $. As an immediate application of the above definition, we next introduce the concept of orthogonality in a vector space equipped with a topology.

\begin{definition}\label{orthogonality}
Let $ X $ be a vector space and let $ \uptau  $ be a topology on $ X $. Let $ \mathcal{A}_X $ be a given $ \rho $-admissible subset of $ X $ and let $ \mathcal{F} $ be a family of scalar valued continuous functions defined on $ (X,\uptau) $. For $ u, v \in \mathcal{A}_X $, we say that $ u\perp_{(\uptau, \mathcal{F},\mathcal{A}_X)} v $ if there exists $ f \in \mathcal{F} $ such that the following two conditions are satisfied:\\

(i) $ f(u) = \underset{z\in \mathcal{A}_X}{sup}|f(z)| $.\\

(ii) $ f(\lambda v) = 0 $ for all scalars $ \lambda $.\\

\noindent For any two non-zero elements $ x, y \in X $, we say that $ x\perp_{(\uptau, \mathcal{F},\mathcal{A}_X)} y $ if $ a_x \perp_{(\uptau, \mathcal{F},\mathcal{A}_X)}  a_y $, where $ a_x $ and $ a_y $ are the corresponding elements of $ x $ and $ y $ in $ \mathcal{A}_X $ respectively. If either of  $ x, y $ is zero we define $ x\perp_{(\uptau, \mathcal{F},\mathcal{A}_X)} y $ and $ y\perp_{(\uptau, \mathcal{F},\mathcal{A}_X)} x $. 
\end{definition}

\noindent The triplet $ (\uptau, \mathcal{F}, \mathcal{A}_X) $ in above definition will be called the orthogonality space of the topological space $ (X,\uptau) $ with respect to the family $ \mathcal{F} $ of scalar valued continuous functions and the $ \rho $-admissible subset $ \mathcal{A}_X $ of $ X $. It would perhaps be more satisfactory to apply the term ``orthogonality space" to the triplet $ {(\uptau, \mathcal{F},\mathcal{A}_X)}  $, rather than to $ (X,\uptau) $. After all, the topological structure of $ (X,\uptau) $  has not been changed in any way by the fact that we now also have a $ \rho $-admissible subset $ \mathcal{A}_X $ of $ X $ and a family of scalar valued continuous functions $ \mathcal{F} $ defined on $ (X, \uptau). $ On the other hand, it is apparent that the newly introduced orthogonality depends heavily on the choice of $ \uptau $, $ \mathcal{F} $ and $ \mathcal{A}_X $. Indeed, we will deal with more of it in the present work. Let us note that the orthogonality $ \perp_{(\uptau, \mathcal{F}, \mathcal{A}_X)} $ is homogeneous. In other words, given any two $ x, y \in X $ and any two non-zero scalars $ \lambda, \mu \in \mathbb{R}, $ we have that $ x\perp_{(\uptau, \mathcal{F},\mathcal{A}_X)} y $ if and only if $ \lambda x\perp_{(\uptau, \mathcal{F},\mathcal{A}_X)} \mu y $. For a given orthogonality space $ {(\uptau, \mathcal{F},\mathcal{A}_X)} $, let $ \mathcal{P}_{(\uptau, \mathcal{F},\mathcal{A}_X)} $ be the subset of $ X \times X $, defined by $ \mathcal{P}_{(\uptau, \mathcal{F},\mathcal{A}_X)} = \{ (x,y) \in X\times X :  x\perp_{(\uptau, \mathcal{F},\mathcal{A}_X)} y \} .$ The set $ \mathcal{P}_{(\uptau, \mathcal{F},\mathcal{A}_X)} $ will be called the orthogonality set of the corresponding orthogonality space $ {(\uptau, \mathcal{F},\mathcal{A}_X)} $. In the same spirit, for a Banach space $ (X, \| \cdot \|) $, we define the Birkhoff-James orthogonality set of $ (X, \| \cdot \|) $ by $ \mathcal{P}_B = \{ (x,y) \in X \times X : x\perp_B y \} .$ Let $  x $ and $ y $ be a pair of non-zero linearly dependent vectors. Then $ x\perp_{(\uptau, \mathcal{F},\mathcal{A}_X)} y $ if and only if there exists a scalar valued continuous function $ f\in \mathcal{F} $ such that $ f $ vanishes on $ \mathcal{A}_X $. We would like to comment here that the orthogonality introduced in the present article may have counter-intuitive properties depending on the choice of $ \uptau $, $ \mathcal{F} $ and $ \mathcal{A}_X $. As for example, let $ \uptau $ be the discrete topology on $ X $ and let $ \mathcal{F} $ be the collection of all non-zero scalar valued functions on $ X $. Then for any $ \rho $-admissible subset $ \mathcal{A}_X $ of $ X $ and for any $ x, y\in X $, it follows that $ x \perp_{(\uptau, \mathcal{F}, \mathcal{A}_X)} y $. We further note that the same conclusion can be reached without any assumption on $ \uptau $ and $ \mathcal{A}_X $ if $ \mathcal{F} $ contains the zero function.\\

However, we will also illustrate in the present work that it is possible to impose certain desirable properties on the orthogonality space $ (\uptau, \mathcal{F}, \mathcal{A}_X) $ by suitably choosing $ \uptau, \mathcal{F}, \mathcal{A}_X $. In particular, we will prove that Birkhoff-James orthogonality in Banach spaces is indeed a special case of the orthogonality defined by us in the present work. We also explore the dependence of the orthogonality $ \perp_{(\uptau, \mathcal{F},\mathcal{A}_X)} $ on both $ \mathcal{F} $ and $ \mathcal{A}_X $ with respect to the norm topology and the weak topology on a Banach space. We completely characterize the right additivity property of $ \perp_{(\uptau, \mathcal{F},\mathcal{A}_X)}, $ under certain additional conditions. In Theorem $ 2.2 $ of \cite{S}, Birkhoff-James orthogonality of linear operators between finite-dimensional real Banach spaces was completely characterized by using the norm attainment set of linear operators. We establish an analogous result for characterizing the orthogonality of continuous linear operators on a finite-dimensional vector space with a topology on it. The applicability of the concerned result is illustrated by the fact that the analogous result in the setting of Banach spaces, as obtained in Theorem $ 2.2 $ of \cite{S}, can be obtained as a corollary to it. Moreover, we illustrate that a topological generalization of the celebrated Bhatia-$\breve{S}$emrl Theorem \cite{BS} can be obtained by following our methodology.

\section{Orthogonality in a vector space with a topology}
We begin with the observation that Birkhoff-James orthogonality in a Banach space is a manifestation of a special case of the orthogonality defined by us in this article. Indeed, this serves as a principal motivation of us behind introducing the generalized concept of orthogonality in the setting of a vector space with a topology.

\begin{theorem}\label{1.P_B equals to P_tau}
Let $ (X, \| \cdot \|) $ be a Banach space and let $ \rho  $ be the projective equivalence relation on $ X\setminus \{ 0 \} $. If $ \mathcal{F} = S_{X^*} $ then there exists a $ \rho $-admissible subset $ \mathcal{A}_X $ of $ X $ such that $ \mathcal{P}_B =  \mathcal{P}_{(\uptau_{\| \cdot \|}, \mathcal{F}, \mathcal{A}_X)} $. In other words, Birkhoff-James orthogonality is equivalent to $ \perp_{(\uptau_{\| \cdot \|}, \mathcal{F}, \mathcal{A}_X)}. $
\end{theorem}

\begin{proof}
Without loss of generality, we may and do assume that each element of $ \mathcal{A}_X $ has norm one, i.e., $ \mathcal{A}_X\subset S_X  $. For a non-zero element $ x\in X $, let $ a_x $ be the corresponding element of $ x $ in $ \mathcal{A}_X $. Let $ (x, y) \in \mathcal{P}_B $. If either of $ x, y $ is zero, then trivially $ (x,y)\in \mathcal{P}_{(\uptau_{\|\cdot \|}, \mathcal{F},\mathcal{A}_X)} $. Similarly, if $ (x,y)\in \mathcal{P}_{(\uptau_{\|\cdot \|}, \mathcal{F},\mathcal{A}_X)} $ and either of $ x, y $ is zero then $ (x,y)\in \mathcal{P}_B $. Now, let $ (x, y)\in \mathcal{P}_{(\uptau_{\|\cdot \|}, \mathcal{F},\mathcal{A}_X)} $ such that $ x $ and $ y $ are non-zero. Then there exists $ f\in \mathcal{F} $ such that $ f(a_x) = \underset{ z\in \mathcal{A}_X }{sup}|f(z)| = 1 $ and $ f(\lambda a_y) = 0 $ for all scalars $ \lambda $. Since $ \mathcal{A}_X $ is a proper subset of $ S_X $, it follows that either $ a_x = \frac{x}{\|x\|} $ or, $ a_x = \frac{-x}{\|x\|} $. Since $ f \in \mathcal{F} = S_{X^*} $, it follows from Theorem $ 2.1 $ of \cite{Ja} that $ a_x\perp_B a_y $. Therefore, using the homogeneity property of Birkhoff-James orthogonality, we have that $ (x,y)\in \mathcal{P}_B $. This completes the proof of the fact that $ \mathcal{P}_{(\uptau_{\| \cdot \|}, \mathcal{F},\mathcal{A}_X)} \subseteq \mathcal{P}_B $. Conversely, suppose, $ (x,y)\in \mathcal{P}_B $ for some non-zero elements $ x $ and $ y $. Then again by the homogeneity property of Birkhoff-James orthogonality, we have $ (\frac{x}{\|x\|}, \frac{y}{\|y\|})\in \mathcal{P}_B $. Therefore, there exists $ f\in S_{X^*} $ such that $ f(\frac{x}{\|x\|}) = 1 $ and $ f(\frac{y}{\|y\|}) = 0 .$ It is trivial to see that $ \frac{x}{\|x\|} \neq \frac{\pm y}{\|y\|} $. Consequently, $ x $ and $ y $ correspond to distinct members in $\mathcal{A}_X $. If $ a_x = \frac{x}{\|x\|} $ then there is nothing to prove. If $ a_x = \frac{-x}{\|x\|} $ then we choose $ -f \in \mathcal{F} $ instead of $ f $. Now, $ (-f)(a_x) = \underset{ z\in \mathcal{A}_X }{sup}|f(z)| = 1 $ and $ (-f)(\lambda a_y) = 0 $ for all scalars $ \lambda $. Therefore, we obtain $ x\perp_{(\uptau_{\| \cdot \|}, \mathcal{F}, \mathcal{A}_X)} y $, i.e., $ \mathcal{P}_B\subseteq \mathcal{P}_{(\uptau_{\| \cdot \|}, \mathcal{F},\mathcal{A}_X)} $. Consequently, $ \mathcal{P}_B = \mathcal{P}_{(\uptau_{\| \cdot \|}, \mathcal{F},\mathcal{A}_X)} $. This completes the proof of the theorem.
\end{proof}

Our next result illustrates that $ \perp_{(\uptau_{\| \cdot \|}, \mathcal{F}, \mathcal{A}_X)} $ can be made strictly weaker than Birkhoff-James orthogonality in a Banach space $ X, $ by suitably choosing $ \mathcal{F} $ and $ \mathcal{A}_X. $

\begin{theorem}\label{2.P_tau proper subset P_B}
Let $ (X, \| \cdot \|) $ be a Banach space and let $\rho  $ be the projective equivalence relation on $ X\setminus \{ 0 \} $. If $ \mathcal{F} = S_{X^*} $, then there exists a $ \rho $-admissible subset $ \mathcal{A}_X $ of $ X $ such that $  \mathcal{P}_{(\uptau_{\| \cdot \|}, \mathcal{F}, \mathcal{A}_X)} \subsetneq \mathcal{P}_B $. In particular, Birkhoff-James orthogonality is not equivalent to $ \perp_{(\uptau_{\| \cdot \|}, \mathcal{F}, \mathcal{A}_X)}. $
\end{theorem}

\begin{proof}
Let us consider $ (u,v)\in \mathcal{P}_B $ for some non-zero $ u, v $. We choose $ \mathcal{A}_X $ in such a way that the element in $ \mathcal{A}_X $ which corresponds to $ u $ has norm $ \frac{1}{2} $ and all other elements of $ \mathcal{A}_X $ has norm $ 1 $.  For a non-zero element $ x\in X $, let $ a_x $ be the corresponding element of $ x $ in $ \mathcal{A}_X $. Since $ \mathcal{F} $ is the family of all norm one continuous linear functionals defined on $ X $, it is easy to see that for any $ g\in \mathcal{F} $, $ | g(a_u) | < \underset{z \in \mathcal{A}_X}{sup}| g(z) | = 1 $. Therefore $ (u,v) \not\in \mathcal{P} _{(\uptau_{\| \cdot \|}, \mathcal{F}, \mathcal{A}_X)} $. Now, let $ (x,y)\in \mathcal{P}_{(\uptau_{\|\cdot \|}, \mathcal{F},\mathcal{A}_X)}. $ If either of $ x, y $ is zero then $ (x,y)\in \mathcal{P}_B $. Let $ (x, y)\in \mathcal{P}_{(\uptau_{\|\cdot \|}, \mathcal{F},\mathcal{A}_X)} $ be such that $ x $ and $ y $ are non-zero elements in $ X $. Then there exists $ f\in \mathcal{F} $ such that $ f(a_x) = \underset{ z\in \mathcal{A}_X }{sup}|f(z)| = 1 $ and $ f(\lambda a_y) = 0 $ for all scalars $ \lambda $. We should note that either $ a_x = \frac{x}{\|x\|} $, or $ a_x = \frac{-x}{\|x\|} $. Since $ f \in \mathcal{F} = S_{X^*} $, it follows from Theorem $ 2.1 $ of \cite{Ja} that $ a_x\perp_B a_y $. Therefore, using the homogeneity property of Birkhoff-James orthogonality, we have that $ (x,y)\in \mathcal{P}_B $. This completes the proof of the fact that $ \mathcal{P}_{(\uptau_{\| \cdot \|}, \mathcal{F},\mathcal{A}_X)} \subset \mathcal{P}_B $. Consequently, it follows that $ \mathcal{P}_{(\uptau_{\| \cdot \|}, \mathcal{F},\mathcal{A}_X)} \subsetneq \mathcal{P}_B $. This completes the proof of the theorem.
\end{proof}

Our next goal is to show that in a Banach space $ X, $ complete description of Birkhoff-James orthogonality in $ X $ is already contained in the weak topology on $ X, $ if we make a natural choice of $ \mathcal{F} $ and $ \mathcal{A}_X. $

\begin{theorem}\label{3.P_B = P_w}
Let $ (X, \| \cdot \|) $ be a Banach space and let $ \uptau_w $ be the weak topology on $ X $. Let $ \rho  $ be the projective equivalence relation on $ X\setminus \{0\} $. Then there exists a family $ \mathcal{F} $ of continuous scalar valued functions and a $ \rho $-admissible subset $ \mathcal{A}_X $ of $ X $ such that $ \mathcal{P}_B =  \mathcal{P}_{(\uptau_{w}, \mathcal{F}, \mathcal{A}_X)} $. In other words, Birkhoff-James orthogonality is equivalent to $ \perp_{(\uptau_w, \mathcal{F}, \mathcal{A}_X)}. $
\end{theorem}
 
\begin{proof}
Without loss of generality, we may and do assume that each element of $ \mathcal{A}_X $ has norm one, i.e., $ \mathcal{A}_X\subset S_X  $. Let us choose $ \mathcal{F} = S_{X^*} $. Now, proceeding in exactly the same way as in Theorem \ref{1.P_B equals to P_tau}, we get the desired result. 
\end{proof}


\begin{remark}
It is well-known that the weak topology on a Banach space $ X $ is not metrizable if $ X $ is infinite-dimensional. Therefore, the above theorem explicitly points out the fact that for building a satisfactory concept of orthogonality, it is not essential to work in a metric setting. 
\end{remark}

Let $ (X, \uptau) $ be a topological space such that every one-point set is closed in $ X $. The space $ (X, \uptau) $ is said to be perfectly normal if for each pair $ A, B $ of disjoint closed sets of $ X $ , there exists a continuous function $ f_{A, B}: X\rightarrow [0, 1] $ such that $ f_{A, B}^{-1}(\{ 0 \}) = B $  and $ f_{A,B}^{-1}(\{ 1 \}) = A $. Such a function $ f_{A, B} $ is said to be a strictly separating function for the disjoint closed sets $ A $ and $ B $. Our next result shows that if $ \uptau $ is a perfectly normal topology on $ (X,\| \cdot \|) $, then under a certain additional condition, Birkhoff-James orthogonality in $ X $ is equivalent to $ \perp_{(\uptau, \mathcal{F}, \mathcal{A}_X)}. $

\begin{theorem}\label{4.P_B Subset P_tau PERFECTLY NORMAL}
Let $ (X, \| \cdot \|) $ be a Banach space and let $ \rho  $ be the projective equivalence relation on $ X\setminus \{ 0 \} $. Let $\uptau $ be any topology on $ X $ such that the following conditions are satisfied:\\ 

(i) $ (X,\uptau) $ is a perfectly normal topological vector space.\\

(ii) Each strictly separating function on $ (X,\uptau) $ is $ \uptau_{\| \cdot \|} $-continuous.\\

\noindent Then for any $ \rho $-admissible subset $ \mathcal{A}_X $ of $ X $, there exists a family $ \mathcal{F} $ of strictly separating functions on $ (X, \uptau) $ such that  $ \mathcal{P}_B =  \mathcal{P}_{(\uptau, \mathcal{F}, \mathcal{A}_X)} $. In other words, Birkhoff-James orthogonality is equivalent to $ \perp_{(\uptau, \mathcal{F}, \mathcal{A}_X)}. $
\end{theorem}

\begin{proof}
For a non-zero element $ x\in X $, let $ a_x $ be the corresponding element of $ x $ in $ \mathcal{A}_X $. For any $ y\in X $ , let $ L_y = \{ \lambda y : \lambda \in \mathbb{R} \} $. It is easy to see that $ L_y $ is closed in both $ ( X, \uptau) $ and $ (X,\uptau_{\|\cdot\|}) $. It should be observed that for any $ (x,y)\in \mathcal{P}_B $ such that $ x $ and $ y $ are non-zero, $ a_x \not\in L_y $. In particular, $ \{ a_x \} $ and $ L_y $ form a disjoint pair of closed sets in $ ( X, \uptau) $. Therefore, there are $ \uptau $-continuous functions which strictly separate $ \{ a_x \} $ and $ L_y $. Let $ \mathcal{F} $ be the collection of all strictly separating functions $ g_{a_x, Ly} $ such that $ (x,y)\in \mathcal{P}_B $ and $ x, y $ are non-zero. In other words,

\[ \mathcal{F} = \{ g_{a_x,L_y} : X \rightarrow [0, 1] : g^{-1}_{a_x,L_y}(\{ 0 \}) = L_y  ,  g^{-1}_{a_x,L_y}(\{ 1 \}) = \{ a_x \},  (x,y)\in \mathcal{P}_B, x,y \neq 0 \}.\]

\noindent By the hypothesis of the theorem, it follows that each member of $ \mathcal{F} $ is also $ \uptau_{\| \cdot \|} $-continuous. Let $ (x,y)\in \mathcal{P}_B $. If either of $ x, y $ is zero then trivially $ (x,y) \in \mathcal{P}_{(\uptau, \mathcal{F}, \mathcal{A}_X)} $. Similarly, if $ (x,y)\in \mathcal{P}_{(\uptau, \mathcal{F},\mathcal{A}_X)} $ and either of $ x,y $ is zero then $ (x,y)\in \mathcal{P}_B $. Let $ (x, y )\in \mathcal{P}_B $ be such that $ x $ and $ y $ are non-zero. Now, $ g_{a_x,L_y} \in \mathcal{F} $. Also, $ g_{a_x,L_y}(a_x) = \underset{z\in \mathcal{A}_X }{sup} |g_{a_x,L_y}(z)| = 1 $ and $ g_{a_x,L_y}(L_y) = 0 $, i.e., $ g_{a_x,L_y}(\lambda y) = 0 $ for all scalars $ \lambda $. Therefore, $ x\perp_{(\uptau, \mathcal{F}, \mathcal{A}_X)} y $. Consequently, $ \mathcal{P}_B \subseteq \mathcal{P}_{(\uptau, \mathcal{F}, \mathcal{A}_X)} $. Now, the equality follows directly from the fact that in the construction of $ \mathcal{F} $, we have restricted ourselves to $ (x, y)\in \mathcal{P}_B $ such that $ x, y \neq 0 $. This completes the proof of the theorem.
\end{proof}

\begin{remark}
We would like to note that the above theorem is another evidence to the fact that the concept of orthogonality should be treated as a topological one, instead of restricting the scope of its study to Banach spaces. We further observe that it is well-known fact that the weak topology on a Banach space $ X $ is not necessarily normal \cite{C}. Therefore, Theorem \ref{3.P_B = P_w} does not follow from the above theorem.   
\end{remark}

If $ (X, \| \cdot \|) $ is a Banach space then it is a topological vector space with respect to the norm topology $ \uptau_{\| \cdot \|} $. Moreover, it is easy to see that $ (X,\uptau_{\| \cdot \|}) $ is a perfectly normal topological vector space. Therefore, as an immediate consequence of above theorem, we have the following corollary.

\begin{cor}\label{5.P_B Subset P_tau NORMAL}
Let $ (X, \| \cdot \|) $ be a Banach space and let $ \rho  $ be the projective equivalence relation on $ X\setminus \{ 0 \} $. Then for any $ \rho $-admissible subset $ \mathcal{A}_X $ of $ X $, there exists a family $ \mathcal{F} $ of strictly separating functions on $ (X, \uptau_{\| \cdot \|}) $ such that $ \mathcal{P}_B =  \mathcal{P}_{(\uptau_{\| \cdot \|}, \mathcal{F}, \mathcal{A}_X)} $. 
\end{cor}


In the next theorem, under an additional condition, we completely characterize the right additivity property of the orthogonality defined by us in the topological setting.

\begin{theorem}\label{6.right additive}
Let $ X $ be a vector space and let $ \uptau $ be a topology on $ X $. Let $ \mathcal{A}_X $ be any $ \rho $-admissible subset of $ X $. Suppose, $ \mathcal{F} $ is a family of non-zero scalar valued continuous linear functionals defined on $ X $ such that for any $ f, g\in \mathcal{F} $, $ f\neq \lambda g $ for any scalar $ \lambda $ with $ \left| \lambda \right| \neq 1 $. Then $ \perp_{(\uptau, \mathcal{F}, \mathcal{A}_X)} $ is right additive if and only if for each element $ a\in \mathcal{A}_X $ there exists at most one $ f\in \mathcal{F} $ such that $ f(a) = \underset{z\in \mathcal{A}_X}{sup}|f(z)| $.
\end{theorem}

\begin{proof}
Let us first prove the sufficient part of the theorem. For each non-zero element $ x\in X $, let $ a_x $ be the corresponding element of $ x $ in $ \mathcal{A}_X $. By our assumption, for each $ a\in \mathcal{A}_X, $ there exists at most one functional $ f\in \mathcal{F} $ such that $ f(a) = \underset{z\in \mathcal{A}_X}{\sup}|f(z)| $. Let $ u, v\in X $ such that $ x \perp_{(\uptau, \mathcal{F}, \mathcal{A}_X)} u $ and $ x \perp_{(\uptau, \mathcal{F}, \mathcal{A}_X)} v $. If either of $ u, v $ is zero then clearly $ x \perp_{(\uptau, \mathcal{F}, \mathcal{A}_X)} (u + v) $. Let us assume that $ u, v $ are non-zero. Therefore, there exist $ f, g\in \mathcal{F} $ such that,\\

$ f(a_x) = \underset{z\in \mathcal{A}_X}{\sup}|f(z)| $ and  $ f(\lambda u) = 0 $ for all scalars $ \lambda $,\\

$ g(a_x) = \underset{z\in \mathcal{A}_X}{\sup}|g(z)| $ and $ g(\sigma v) = 0 $ for all scalars $ \sigma $.\\

It follows from our assumption that $ f = g $. Therefore, $ \sigma v \in ker~f $ for all scalars $ \sigma $. However, this implies that $ f(\mu (u+v)) = 0 $ for all scalars $ \mu $ and $ x \perp_{(\uptau, \mathcal{F}, \mathcal{A}_X)} (u+v) .$ In other words, $ \perp_{(\uptau, \mathcal{F}, \mathcal{A}_X)} $ is right additive. We now prove the necessary part of the theorem. Let $ a\in \mathcal{A}_X $ be arbitrary. If possible, suppose that there exist distinct $ f,g\in \mathcal{F} $ such that $ f(a) = \underset{z\in \mathcal{A}_X}{\sup}|f(z)| $ and $ g(a) = \underset{z\in \mathcal{A}_X}{\sup}|g(z)| $. Clearly, $ f \neq -g $. In addition, it follows from the hypothesis of the theorem that $ f\neq \lambda g $ for any scalar $ \lambda $ with $ \left| \lambda \right| \neq 1 $. In particular, $ ker~f \neq ker~g $. Thus $ X = \{ u + v : u\in ker~f, ~v\in ker~g\} = ker~f + ker~g $. Therefore, for any $ w\in X $, $ w = u+v $ for some $ u\in ker~f $ and $ v\in ker~g $. Since $ f, g $ are linear, it follows that $ f(\lambda u) = 0 $ and $ g(\sigma v ) = 0 $ for all scalars $ \lambda $ and $ \sigma $. Therefore, $ a \perp_{(\uptau, \mathcal{F}, \mathcal{A}_X)} u $ and $ a \perp_{(\uptau, \mathcal{F}, \mathcal{A}_X)} v $. Since $ \perp_{(\uptau, \mathcal{F}, \mathcal{A}_X)} $ is right additive, we have $ a \perp_{(\uptau, \mathcal{F}, \mathcal{A}_X)} (u+v) $. This shows that  $ a \perp_{(\uptau, \mathcal{F}, \mathcal{A}_X)} w $ for all $ w\in X $. In particular, $ a \perp_{(\uptau, \mathcal{F}, \mathcal{A}_X)} a $. Since $ a\in \mathcal{A}_X $, it follows that $ a \neq 0 $. So, there exists $ h\in \mathcal{F} $ such that\\

$ h(a) = \underset{z\in \mathcal{A}_X}{sup}|h(z)| $ and  $ h(\lambda a) = 0 $ for all scalars $ \lambda $.\\

Clearly, this can be true only when $ h(\mathcal{A}_X) = 0 $, or, equivalently, when $ h $ is identically zero on $ X. $ Since $ \mathcal{F} $ does not contain the zero functional, this leads us to a contradiction. This completes the proof of the theorem.
\end{proof}


It was proved in Theorem $ 4.2 $ of \cite{Ja} that Birkhoff-James orthogonality is right additive in a Banach space $ X $ if and only if $ X $ is smooth, i.e., there exists a unique supporting hyperplane to the unit ball of $ X $ at every point of the unit sphere of $ X. $ This characterization of smoothness of a Banach space in terms of the right additivity property of Birkhoff-James orthogonality is particularly useful in identifying the smooth points in the Banach space of bounded linear operators, endowed with the usual operator norm \cite{SPMR}. We next prove that the above characterization of the smoothness of a Banach space can actually be obtained as a corollary to our previous theorem.

\begin{cor}\label{7.Banach right additive}
Let $ (X,{\| \cdot \|}) $ be a Banach space. Then Birkhoff-James orthogonality is right additive in $ X $ if and only if $ (X, \| \cdot \|) $ is smooth. 
\end{cor}

\begin{proof}
Let us consider the topological vector space $ (X, \uptau_{\| \cdot \|}) $. Without loss of generality, we assume that $ \mathcal{A}_X $ is a $ \rho $-admissible subset of $ X $ such that $ \mathcal{A}_X \subset S_{X}. $ Let $ \mathcal{F} = S_{X^*} $. It is easy to see that the criteria of Theorem \ref{6.right additive} is satisfied in this setting. 
We first prove the necessary part of the corollary. Suppose, Birkhoff-James orthogonality is right additive in $ X. $  Now from Theorem \ref{1.P_B equals to P_tau}, we have $ \mathcal{P}_{(\uptau_{\| \cdot \|}, \mathcal{F}, \mathcal{A}_X)} = \mathcal{P}_B $. It follows from the necessary part of Theorem \ref{6.right additive} that for each $ a\in\mathcal{A}_X, $ there exists at most one $ f\in \mathcal{F} $ such that $ f(a) = \underset{z\in \mathcal{A}_X}{sup}|f(z)| = 1 $. However, this is clearly equivalent to the fact that $ (X, \| \cdot \| ) $ is smooth. We now prove the sufficient part of the corollary. Suppose, $ (X, \| \cdot \| ) $ is smooth. Then for each $ u\in S_X ,$ there exists exactly one $ f\in \mathcal{F} $ such that $ f(u) = \| f \| = 1 $. In particular, for each $ a\in \mathcal{A}_X $ there exists exactly one $ f\in \mathcal{F} $ such that $ f(a) = \underset{z\in \mathcal{A}_X}{sup}|f(z)| = 1 $. Now, by the sufficient part of Theorem \ref{6.right additive}, $ \perp_{(\uptau_{\| \cdot \|}, \mathcal{F}, \mathcal{A}_X)} $ is right additive. Once again from Theorem \ref{1.P_B equals to P_tau}  it follows that Birkhoff-James orthogonality is right additive. This completes the proof of the sufficient part of the corollary.
\end{proof} 
\section{Orthogonality of linear operators And a generalization of Bhatia-$\breve{S}$emrl Theorem}

In \cite{BS}, Bhatia and $\breve{S}$emrl studied the Birkhoff-James orthogonality of linear operators on finite-dimensional Hilbert spaces. Let $ H $ be an $ n $-dimensional Hilbert space. Let $ A $ and $ B $ be $ n \times n $ matrices, identified as linear operators, acting on $ H $ in the usual way. A complete characterization of the Birkhoff-James orthogonality of linear operators on finite-dimensional Hilbert spaces was obtained in \cite{BS} by means of the following theorem:

\begin{theorem}[Theorem $ 1.1 $  \cite{BS}]\label{Bhatia}
A matrix $ A $ is orthogonal to $ B $ if and only if there exists a unit vector $ x\in H $ such that $ \| Ax\| = \| A \| $ and $ \langle Ax, Bx \rangle = 0 $.
\end{theorem}

Later on, in \cite{S}, Sain characterized the Birkhoff-James orthogonality of linear operators between finite-dimensional real Banach spaces by introducing the notion of the positive part of $ x $, denoted by $ x^+ $ and the negative part of $ x $, denoted by $ x^- $, for an element $ x $ in a real Banach space. Let $ X $ be a finite-dimensional real Banach space and let $ x\in X $. For any element $ y \in X $, we say that $ y\in x^+ $, if $ \| x+\lambda y\| \geq \| x \| $ for all $ \lambda \geq 0 $. Accordingly, we say that $ y \in x^- $,  if $ \| x + \lambda y\| \geq \| x \| $ for all $ \lambda \leq 0.$ Let $ T\in \mathbb {L}(X) $, the collection of all bounded linear operators from the Banach space $ X $ to itself. Let $ M_T $ be the norm attainment set of $ T $, i.e., $ M_T = \{ x\in S_X : \| Tx\| = \| T \| \} $. For the convenience of the readers, we quote the aforesaid characterization due to Sain:

\begin{theorem}[Theorem $ 2.2 $ \cite{S}]\label{Debmalya Sain}
Let $ X $ be a finite-dimensional real Banach space. Let $ T, A \in \mathbb{L}(X) $. Then $ T \perp_ B A $ if and only if there exist $ x, y \in M _T $ such that $ Ax \in T x^+ $ and $ Ay \in T y^- $. .
\end{theorem}

Theorem \ref{Bhatia} as well as Theorem \ref{Debmalya Sain} show that for a linear operator $ T\in \mathbb{L}(X) $ the sets $ T^{\perp} = \{ A\in \mathbb{L}(X) : T\perp_B A \} $ and $ M_T $ share a deep relation. Motivated by this observation, we strive for building an analogous theory in the topological setting. We would like to end the present article by accomplishing the said goal in a special yet sufficiently general case. However, we require some preparations before embarking on such a journey. First we start with a definition that breaks orthogonality into two components and then we will move on to describe the prescribed topologies, the class of continuous scalar valued functions and the $ \rho $-admissible subsets of the spaces $ X $, $ Y $ and $ \mathbb{L}(X,Y) $. 

\begin{definition}\label{x  plus and x  minus}
Let $ X $ be a vector space and $ \uptau $ be a topology on $ X $. Let $ \mathcal{A}_X $ be any $ \rho $-admissible subset of $ X $ and let $ \mathcal{F} $ be a family of continuous scalar valued functions on $ (X,\uptau) $. Let $ x,y $ be any two non-zero elements in $ X $.\\
First suppose that $ x = \mu a_x $ for some $ \mu > 0 $ and for some $ a_x \in \mathcal{A}_X. $ We say that $ y\in x^{\oplus} $ if there exists $ f\in \mathcal{F} $ such that $ f(a_x) = \underset{z\in \mathcal{A}_X}{\sup}|f(z)| $ and $ f(\lambda y) \geq 0 $ for all $ \lambda \geq 0 $.  Similarly, we say that $ y\in x^{\ominus} $ if there exists $ h\in \mathcal{F} $ such that $ h(a_x) = \underset{z\in \mathcal{A}_X}{\sup}|h(z)| $ and $ h(\lambda y) \leq 0 $ for all $ \lambda \geq 0 $.\\
Next, suppose that $ x = \mu a_x $ for some $ \mu < 0 $ and for some $ a_x \in \mathcal{A}_X. $ We say that $ y\in x^{\oplus} $ if there exists $ g\in \mathcal{F} $ such that $ g(a_x) = \underset{z\in \mathcal{A}_X}{\sup}|g(z)| $ and $ g(\lambda y) \leq 0 $ for all $ \lambda \geq 0 $.  Similarly, we say that $ y\in x^{\ominus} $ if there exists $ k\in \mathcal{F} $ such that $ k(a_x) = \underset{z\in \mathcal{A}_X}{\sup}|k(z)| $ and $ k(\lambda y) \geq 0 $ for all $ \lambda \geq 0 $.\\
If either of $ x,y $ is zero, we declare $ y \in x^{\oplus}\cap x^{\ominus} $ and $ x \in y^{\oplus}\cap y^{\ominus} $.
\end{definition}

 Let $ X $ be a finite-dimensional vector space equipped with some topology $ \uptau $. We further assume that $ \uptau $ is such that we can choose a $ \rho $-admissible subset $ \mathcal{A}_X $ of $ X $ with $ \mathcal{A} = ((\mathcal{A}_X) \cup (- \mathcal{A}_X ) ) $ is compact. Let $ \mathcal{F} $ be any family of continuous scalar valued functions defined on $ X $.
 
Let $ Y $ be a vector space and let $ \{p_i\}_{i=1}^m $ be a finite family of non-trivial semi-norms on $ Y $, i.e., there exist $ \{y_i\}_{i=1}^m \subset Y $, such that $ p_i(y_i) \neq 0 $ for each $ 1\leq i \leq m $. We topologize $ Y $ in the following well-known and standard way \cite{W}:

\noindent Associate to each $ p \in \{ p_i\}_{i=1}^m $ and to each positive integer $ n $ the set $ V_{(p,n)} = \{ y\in Y: p(y)< \frac{1}{n} \} $. Let $ \gamma $ be the collection of all finite intersections of the sets $ V_{(p,n)} $. Let $ \Gamma $ be a collection of subsets of $ Y $ such that $ B \in \Gamma $ if and only if $ B $ is a union of the translates of members of $ \gamma $. It is easy to see that $ \Gamma $ is a topology on $ Y $ and each $ p\in \{ p_i \}_{i=1}^m $ is continuous under $ \Gamma $. Moreover, vector addition and scalar multiplication are also continuous under $ \Gamma $. We refer the readers to Chapter 1 of \cite{W} for more information in this regard. Let $ \mathcal{A}_Y $ be a $ \rho $-admissible subset in $ Y $ such that for any $ b\in \mathcal{B} = ((\mathcal{A}_Y) \cup (- \mathcal{A}_Y ) ) $, $ \max \{ p_i(b)\}_{i=1}^m\in \{ 0, 1 \} $. We note that such a choice of $ \mathcal{B} $ is always possible. For a given element $ b\in \mathcal{B} $, we say $ p\in \{ p_i\}_{i=1}^m $ is optimal for $ b $ if $ p(b) = 1 $. Let $ W_b = \{ p\in \{ p_i\}_{i=1}^m : p(b)= 1\} $. It is immediate that for each $ b\in \mathcal{B} $ such that $ \max \{ p_i(b)\}_{i=1}^m = 1 $, there exists at least one optimal semi-norm for $ b $, i.e., $ W_b\neq \emptyset $. Let us define $ \mathcal{S}_{Y} = \{ b\in \mathcal{B} : p(b) = 1, p\in \{ p_i\}_{i=1}^m \} $. Clearly, $ \mathcal{S}_{Y} $ is non-empty. Let $ b\in \mathcal{S}_Y $. We claim that for any $ p\in W_b $, there exists a linear functional $ f_{(p,b)} : Y \longrightarrow \mathbb{R} $ such that $ f_{(p,b)}(b) = p(b)= 1 $ and $ |f_{(p,b)}(z)| \leq | p(z) | $ for all $ z\in Y $. In other words, our claim is that there exists a $ p $-dominated linear functional $ f_{(p,b)} : Y \longrightarrow \mathbb{R} $ such that $ f_{(p,b)}(b) = p(b) = 1 $. Let us fix some $ p\in W_b $. Let us define $ f: span\{ b \} \longrightarrow \mathbb{R} $ by $ f(\mu b ) = \mu p(b) $. Clearly, $ f $ is linear and $ f $ is dominated by $ p $ in $ span \{ b \}. $ It is now easy to deduce that $ f $ is continuous. Therefore, $ f $ possesses a continuous linear extension, say, $ f_{(p,b)} : Y \longrightarrow \mathbb{R} $ such that $ | f_{(p,b)}(z)| \leq p(z) $ for all $ z\in Y $ and $ f_{(p,b)}(b) = 1 $. Let $ \mathcal{S}_{Y^*} $ be the collection of all such linear functionals. In other words,
\[ \mathcal{S}_{Y^*} = \{ f_{(p,b)}: |f_{(p,b)}(z)| \leq p(z)~ \forall z\in Y,~ f_{(p,b)}(b) = 1,~ b\in \mathcal{S}_Y,~  p\in W_b  \} .\]

\noindent Now, we define a semi-norm $ P $ on $ \mathbb{L}(X,Y) $, the collection of all continuous linear operators from $ X $ to $ Y $, by
\[  P(T) = \underset{i}{\max}\{ \underset{a\in \mathcal{A}}{\max}~ p_i(T(a))\}. \] 

\noindent We call $ P $ as the semi-norm induced by the family $ \{p_i\}_{i=1}^m $. We should note that the semi-norm $ P $ is a nontrivial semi-norm on $ \mathbb{L}(X,Y) $. We topologize $ \mathbb{L}(X,Y) $ by the singleton family of semi-norm $ \{ P \} $ in the same way as we have topologized $ Y $ by the family of semi-norms $ \{ p_i \}_{i=1}^m $. We denote the vector space $ \mathbb{L}(X, Y) $, topologized in this way, as $ (\mathbb{L}(X, Y), P) $. Since $ \mathcal{A} $ is compact in $ X $, by continuity of $ T $, $ T(\mathcal{A}) $ is compact in $ Y $. Also, by continuity of each $ p\in \{p_i\}_{i=1}^m $, $ p(T(\mathcal{A})) $ is compact in $ \mathbb{R} $. Moreover, for each $ T \in \mathbb{L}(X, Y) $ and for each $ p\in \{p_i\}_{i=1}^m $, $ p $ attains supremum on $ T(\mathcal{A}) $. 
For each $ T\in \mathbb{L}(X, Y), $ we define
\[ \mathcal{M}_T = \{ a\in \mathcal{A} : P(T) = p(T(a)), ~\textit{for ~ some}~p\in \{p_i\}_{i=1}^m \}. \]

We also require the following definition to serve our purpose. Let
\[ \mathcal{P}_T = \{ p\in \{p_i\}_{i=1}^m : P(T) = p(T(a)), ~ \textit{for ~ some} ~ a \in \mathcal{A} \}. \]

\noindent Clearly, for each $ T\in \mathbb{L}(X, Y) $, $ \mathcal{M}_T \neq \emptyset $ and $ \mathcal{P}_T \neq \emptyset $. Without loss of generality, we may and do choose a $ \rho $-admissible subset $ \mathbb{A}_{\mathbb{L}(X, Y)} $ in $ \mathbb{L}(X, Y) $ such that for each $ T\in  \mathbb{A} = ((\mathbb{A}_{\mathbb{L}(X, Y)}) \cup (-\mathbb{A}_{\mathbb{L}(X, Y)})) $, $ P(T) \in \{ 0, 1 \} $. Let $ \mathcal{S}_{\mathbb{L}(X, Y)} = \{ T\in \mathbb{A} : P(T) = 1 \} $. Now, for any $ T \in \mathcal{S}_{\mathbb{L}(X, Y)} $, $ W_T = \{ P \} $. As argued before, for each $ T\in  \mathcal{S}_{\mathbb{L}(X, Y)} $, there exists a $ P $-dominated linear functional $ F_T $ such that $ F_T(T) = P(T) = 1 $. Let $ \mathcal{S}_{\mathbb{L}(X, Y)^*} $ be the collection of all such functionals. In other words,\\
\noindent \[ \mathcal{S}_{\mathbb{L}(X,Y)^*} = \{ F_T: |F_T(A)| \leq P(A)~ \forall A\in \mathbb{L}(X, Y),~F_T(T) = 1,~ T\in \mathcal{S}_{\mathbb{L}(X, Y)}\} .\]

We next introduce another definition which is intimately related to Definition \ref{x  plus and x  minus}. The only reason behind introducing such definition is to make our further treatment look more convenient.

\begin{definition}\label{oplus p ominus p}
Let $ (Y,\Gamma) $ be a vector space topologized by a finite family of semi-norms $ \{p_i\}_{i=1}^m $. Let $ x\in Y $ be non-zero. Let $ b_x $ be the corresponding element of $ x $ in $ \mathcal{A}_Y $. Let $ p\in W_{b_x} $. Let $ y\in Y $. \\
First suppose that $ x = \mu b_x $ for some positive scalar $ \mu $. We say that $ y\in x^{\oplus_p} $, if there exists a $ p $-dominated linear functional $ f_{(p,b_x)}\in \mathcal{S}_{Y^*} $ such that $ f_{(p,b_x)}(b_x) = \underset{z\in \mathcal{A}_Y}{\sup}|f_{(p,b_x)}(z)| $ and $ f_{(p,b_x)}(y) \geq 0 $.  Similarly, we say that $ y\in x^{\ominus_p} $, if there exists a $ p $-dominated linear functional $ h_{(p,b_x)}\in \mathcal{S}_{Y^*} $ such that $ h_{(p,b_x)}(b_x) = \underset{z\in \mathcal{A}_X}{\sup}|h_{(p,b_x)}(z)| $ and $ h_{(p,b_x)}(y) \leq 0 $.\\
Next, suppose that $ x = \mu b_x $ for some negative scalar $ \mu $. We say that $ y\in x^{\oplus_p} $, if there exists a $ p $-dominated linear functional $ g_{(p,b_x)}\in \mathcal{S}_{Y^*} $ such that $ g_{(p,b_x)}(b_x) = \underset{z\in \mathcal{A}_X}{\sup}|g_{(p,b_x)}(z)| $ and $ g_{(p,b_x)}(y) \leq 0 $. Similarly, we say that $ y\in x^{\ominus_p} $, if there exists a $ p $-dominated linear functional $ k_{(p,b_x)}\in \mathcal{S}_{Y*} $ such that $ k_{(p,b_x)}(b_x) = \underset{z\in \mathcal{A}_X}{\sup}|k_{(p,b_x)}(z)| $ and $ k_{(p,b_x)}(y) \geq 0 .$\\
If $ x $ is zero, we declare $ y\in x^{\oplus_p}\cap x^{\ominus_p} $ and $ x\in y^{\oplus_p}\cap y^{\ominus_p} $, for all $ y\in Y $ and for all $ p\in \{ p_i\}_{i=1}^m $.
\end{definition}

\begin{theorem}\label{9.semi norm orthogonality}

Let $ (Y,\Gamma) $ be the vector space topologized by a finite family of semi-norms $ \{p_i\}_{i=1}^m $ as above. Let $ \mathcal{A}_Y $ be the $ \rho $-admissible subset of $ Y $ defined as above. Let $ \mathcal{F} = \mathcal{S}_{Y^*} $. Let $ u \in Y\setminus \{0\} $ with $ \underset{i}{\max}\{ p_i(u) \}_{i=1}^m \neq 0 $ and let $ c\in \mathcal{A}_Y $ be such that $ u = \sigma c $, for some $ \sigma \in \mathbb{R} $. Let $ p\in W_c $. Then for any $ v\in Y $, $ v\in u^{\oplus_p} $ $ ( u^{\ominus_p}) $ if and only if $ p(u+\lambda v) \geq p(u) $ for all $ \lambda \geq 0 $ $ ( \leq 0 ) $. 

\end{theorem}

\begin{proof}

Let us first prove the necessary part of the theorem. Let $ v\in u^{\oplus_p} $, for some $ v\in Y $ and for some $ p\in W_c $. Without loss of generality, we may assume that $ \sigma > 0 $. Therefore, there exists a $ p $-dominated linear functional $ f_{(p,c)} \in \mathcal{S}_{Y^*} $ such that $ f_{(p,c)}(c) = \underset{w\in \mathcal{A}_Y}{\sup}|f(w)| = p(c) = 1 $ and $ f_{(p,c)}(v)\geq 0 $. Now, for any scalar $ \lambda \geq 0 $ we have,
\[ p(u+ \lambda v) \geq f_{(p,c)}(u+ \lambda v) = f_{(p,c)}(u) + \lambda f_{(p,c)}(v) \geq f_{(p,c)}(u) = \sigma f_{(p,c)}(c) = p(u) .\]

 We now prove the sufficient part of the theorem. Suppose, $ v\in Y $ with some $ p\in W_c $ such that $ p(u+\lambda v) \geq p(u) $ for all $ \lambda \geq 0 $. Without loss of generality, we may assume that $ \sigma > 0 $. Now, we consider the following two cases:\\

{ Case I:} $ p(u+\lambda v) \geq p(u) $ for all scalars $ \lambda \geq 0 $ and for some negative scalar $ \lambda_0 $, $ p(u+ \lambda_0 v) < p(u) $.

 We define $ f: span\{ c \} \longrightarrow \mathbb{R} $ by $ f(\mu c ) = \mu p(c) $. Clearly, $ f $ is linear, dominated by $ p $ in $ span \{ c \} $ and hence $ f $ is continuous. Therefore, $ f $ possesses a continuous linear extension $ f_{(p,c)} : Y \longrightarrow \mathbb{R} $ such that $ | f_{(p,c)}(z)| \leq p(z) $ for all $ z\in Y $. Consequently, $ f_{(p,c)}(c) = \underset{w\in \mathcal{A}_Y}{\sup}|f_{(p,c)}(w)| = p(c) = 1 $. In other words, $ f_{(p,c)} \in \mathcal{S}_{Y^*} $. Now,
\[ p(u) > p(u + \lambda_0 v) \geq  f_{(p,c)}(u) + \lambda_0 f_{(p,c)}(v). \]

Since $ f_{(p,c)}(u) = p(u) $ and $ \lambda_0< 0 $, it follows that $ f_{(p,c)}(v) > 0 $.\\

{ Case II:} $ p(u+\lambda v) \geq p(u) $ for all scalars $ \lambda  $. 

It is easy to see that $ u $ and $ v $ must be linearly independent. We define $ f: span\{ c, v \} \longrightarrow \mathbb{R} $ by $ f( \mu c + \kappa  v) = \mu p (c) $. Now, for any non-zero scalar $ \mu $ we have,
\[ | f(\mu c + \kappa v ) | = | \mu p(c) |  = | \mu | p( c)  \leq | \mu | p( c + \frac{\kappa}{\mu} v)  = p( \mu  c + \kappa v ) .\]

\noindent Clearly, $ f $ is linear, $ f(v) = 0 $ and $ f $ is dominated by $ p $ in $ span \{ c, v \} $. Hence $ f $ is continuous and possesses a continuous linear extension $ f_{(p,c)} : Y \longrightarrow \mathbb{R} $ such that $ | f_{(p,c)}(z)| \leq p(z) $ for all $ z\in Y $. Consequently, $ f_{(p,c)}(c) = \underset{w\in \mathcal{A}_Y}{\sup}|f_{(p,c)}(w)| = p(c) = 1 $ and $ f_{(p,c)} \in \mathcal{S}_{Y^*} $. Using analogous techniques, we can similarly prove that $ v\in u^{\ominus_p} $, for some $ p\in W_c, $ if and only if $ p(u+\lambda v) \geq p(u) $ for all $ \lambda \leq 0 $. This completes the proof of the theorem.  

\end{proof}

The above theorem, in fact, provides some more information regarding the orthogonality space $ (\Gamma, \mathcal{S}_{Y^*}, \mathcal{A}_Y) $. We record the following obvious yet useful observation as a corollary to the above theorem. The proof of the corollary is omitted as it is trivial in view of Theorem \ref{9.semi norm orthogonality}.

\begin{cor}\label{10.semi norm equivalent orthogonality}

Let $ (Y,\Gamma) $ be the topological space and let $ \mathcal{A}_Y $ be the $ \rho $-admissible subset of $ Y $ defined as above. Let $ \mathcal{F} = \mathcal{S}_{Y^*} $. Let $ u \in Y\setminus \{0\} $ be such that $ b_u\in \mathcal{S}_Y $. Then for any $ v\in Y ,$ there exists $ p\in W_{b_u} $ such that the following three conditions are equivalent:\\
(i) $ u \perp _{(\Gamma, \mathcal{F}, \mathcal{A}_Y)} v $.\\
(ii) $ v\in u^{\oplus_p} \cap u^{\ominus_p} $.\\
(iii) $ p(u+\lambda v) \geq p(u) $ for all scalars $ \lambda $.

\end{cor}

As an application of the above theorem, we now obtain a complete characterization of the orthogonality of continuous linear operators in the desired setting.
 
\begin{theorem}\label{11.characterization}

Let $ (Y,\Gamma) $ be the vector space topologized by a finite family of semi-norms $ \{p_i\}_{i=1}^m $ as above. Let $ (X, \uptau),~ \mathcal{A}_X $ be defined as above. Let $ P $ be the semi-norm on $ \mathbb{L}(X, Y) $, induced by the family $ \{ p_i\}_{i=1}^m $. Let $ \mathcal{F} = \mathcal{S}_{\mathbb{L}(X,Y)^*} $. Let $ T, A \in \mathbb{L}(X, Y) $ with $ P(T) \neq 0 $. Then $ T\perp_{(P, \mathcal{S}_{\mathbb{L}(X, Y)^*}, \mathbb{A}_{\mathbb{L}(X, Y)})} A $ if and only if there exist $ x, y \in \mathcal{M}_T $ with $ p, q \in \mathcal{P}_T $ such that $ p(T(x)) = q(T(y)) = P(T) $ and $  Ax\in Tx^{\oplus_p} $, $  Ay\in Ty^{\ominus_q}. $

\end{theorem}
 
\begin{proof}
If $ A $ is zero then the statement of the theorem holds trivially from the corresponding definitions. Let $ T, A \in \mathbb{L}(X, Y) $ be non-zero. Let us first prove the sufficient part of the theorem. Since $ Ax\in Tx^{\oplus_p} $, therefore, for any scalar $ \lambda \geq 0 $, $ P(T+ \lambda A) \geq p(Tx + \lambda Ax) \geq p(Tx) = P(T) $. Similarly, since $  Ay\in Ty^{\ominus_q} $, therefore, for any scalar $ \lambda \leq 0 $, $ P(T+ \lambda A) \geq q(Ty + \lambda Ay) \geq q(Ty) = P(T) $. Applying Corollary \ref{10.semi norm equivalent orthogonality} on $ (P,\mathcal{S}_{\mathbb{L}(X,Y)^*},\mathbb{A}_{\mathbb{L}(X,Y)}) $, we have $ T\perp_{(P, \mathcal{S}_{\mathbb{L}(X, Y)^*}, \mathbb{A}_{\mathbb{L}(X, Y)})} A $. We next prove the necessary part of the theorem. Let $ T, A \in \mathbb{L}(X, Y) $ be non-zero. If possible, suppose, the statement is not true. Now, given any $ x  \in \mathcal{M}_T $ and any $ p\in \mathcal{P}_T $ with $ P(T) = p(T(x)) $, it is easy to see that either $  Ax\in Tx^{\oplus_p} $, or $  Ax\in Tx^{\ominus_p}. $ Therefore, without loss of generality, we may assume that for each $ x\in \mathcal{M}_T $ and for each $ p\in \mathcal{P}_T $ with $ P(T) = p(T(x)) $, $ Ax\in Tx^{\oplus_p} $ and $ Ax\notin Tx^{\ominus_p} $. Consider any $ x\in \mathcal{M}_T $. For each $ p_i\in \{ p_i \}_{i=1}^m $, consider the function $ g^i_x : \mathcal{A} \times [-1, 1] \longrightarrow \mathbb{R} $ defined by 
\[ g^i_x(u, \lambda) = p_i(T u + \lambda Au). \]

\noindent It is straightforward to check that $ g^i_x $ is continuous. Since for every $ i\in \{ 1,2, \dots, m \} $, there exists $ \lambda_{i,x} < 0 $ such that $ p_i(Tx + \lambda_{i,x} Ax) < P(T) $, it follows that $ g^i_x(x, \lambda) < P(T) $. Therefore, by continuity of $ g^i_x, $ there exists an open set $ V_{i,x} $ containing $ x $, in the subspace topology of $ \mathcal{A} $ and $ \delta_{i,x} > 0 $ such that $ g^i_x(w, \lambda) < P(T) $ for each $ w \in V_{i,x} $ and for each $ \lambda \in (\lambda_{i,x} - \delta_{i,x} , \lambda_{i,x} + \delta_{i,x} ).$ Using convexity property of the semi-norm function, it is easy to show that $ g^i_x(w, \lambda) = p_i(T w + \lambda Aw ) < P(T) $ for all $ w \in V_{i,x} $ and for all $ \lambda\in (\lambda_{i,x} , 0).$ Let $ V_x = \bigcap\limits_{i=1}^m V_{i,x} $ and let $ \lambda_x = \frac{1}{2}\min~ \{ \lambda_{i,x} \} $. For any $ z \in \mathcal{A} \setminus \mathcal{M}_T $, we have $ g^i_x(z, 0) = p_i(T z) < P( T ) $. Thus by continuity of $ g^i_x $, there exists an open set $ V_{i,z} $ containing $ z $, in the subspace topology of $ \mathcal{A} $ and $ \delta_{i,z} > 0 $ such that $ g^i_x(y, \lambda) = p_i(T y + \lambda Ay) < P (T) $ for all $ y \in V_{i,z} $ and for all $ \lambda \in (-\delta_{i,z} , \delta_{i,z} ). $ Let $ V_z = \bigcap\limits_{i=1}^m V_{i,z} $ and $ \delta_z = \frac{1}{2}\min~\{ \delta_{i,z} \} $. Clearly, $ \{ V_x : x \in \mathcal{M}_T \} \cup \{ V_z : z\in \mathcal{A} \setminus \mathcal{M}_T \} $ forms an open cover of $ \mathcal{A} $ . Since $ \mathcal{A} $ is compact, this open cover admits a finite sub-cover. Therefore,
\[  \mathcal{A} \subseteq ((\bigcup\limits_{r = 1}^{k_1} V_{x_r} ) \cup (\bigcup\limits_{s = 1}^{k_2} V_{z_s})) ,\]

\noindent for some natural numbers $ k_1 $ and $ k_2 $. Choose $ \lambda _ 0 \in (\bigcap\limits_{r = 1}^{k_1}(\lambda_{x_r }, 0 )) \cap  (\bigcap\limits_{s = 1}^{k_2}(- \delta_{z_s} ,  \delta_{z_s }))  .$ Since $ \mathcal{A} $ is compact, $ \mathcal{M}_{T + \lambda_ 0 A} \neq \emptyset $. Let $ w_0 \in \mathcal{M}_{T + \lambda_ 0 A} $ and $ \hat p\in \mathcal{P}_{T + \lambda_0 A} $. Then either $ w_0 \in V_{x_r} $ for some $ x_r \in \mathcal{M}_ T $,  or $ w _0 \in V_{z_s} $ for some $ z _s \in  \mathcal{A}\setminus \mathcal{M}_T $ . In either case, it follows from the choice of $ \lambda_ 0$  that $ P( T +\lambda_0 A) = \hat p((T +\lambda_0 A)w_0)  < P( T ) $. Now, applying Corollary \ref{10.semi norm equivalent orthogonality} on $  (P,\mathcal{S}_{\mathbb{L}(X,Y)^*},\mathbb{A}_{\mathbb{L}(X,Y)}) $, we get a contradiction to our primary assumption that $ T\perp_{(P, \mathcal{S}_{\mathbb{L}(X, Y)^*}, \mathbb{A}_{\mathbb{L}(X, Y)})} A $. This proves the necessary part of the theorem and thereby establishes it completely.

\end{proof}

Our final result of the present article is the observation that Theorem $ 2.2 $ of \cite{S} can be obtained as a corollary to the above theorem. Given any two Banach spaces $ X $ and $ Y $, we use the same symbol $ \mathbb{L}(X,Y) $ to denote the Banach space of all continuous linear operators from $ X $ to $ Y $, endowed with the usual operator norm.

\begin{cor}\label{12. operator orthogonality}
Let $ X $ be a finite-dimensional Banach space and let $ Y $ be any Banach space. Let $ T, A \in \mathbb{L}(X, Y) $. Then $ T\perp_B A $ if and only if there exist $ x, y \in M_T $ such that $ Ax\in Tx^+ $ and $ Ay\in Ty^- $.
\end{cor}

\begin{proof}
Let $ \mathcal{A}_X $ be a $ \rho $-admissible subset of $ X $ such that each element of $ \mathcal{A}_X $ has norm one, i.e., $ \mathcal{A}_X \subset S_X $. It is trivial to see that $ ((\mathcal{A}_X)\cup (-\mathcal{A}_X)) = S_X $. Since $ X $ is finite-dimensional, $ S_X $ is compact. Let $ p $ be the norm associated to the Banach space $ Y $ and let $ P $ be the norm in $ \mathbb{L}(X, Y) $ induced by $ p $. Let $ \mathcal{A}_Y $ be a $ \rho $-admissible subset of $ Y $ and let $ \mathbb{A}_{\mathbb{L}(X, Y)} $ be a $ \rho $-admissible subset of $ \mathbb{L}(X, Y) $. Without loss of generality, we choose $ \mathcal{A}_Y $ and $ \mathcal{A}_{\mathbb{L}(X, Y)} $ to be such that $ \mathcal{A}_Y \subset S_Y $ and $ \mathcal{A}_{\mathbb{L}(X, Y)} \subset S_{\mathbb{L}(X, Y)} $. It is trivial to see that $ ((\mathcal{A}_Y)\cup (-\mathcal{A}_Y)) = S_Y $ and $ ((\mathcal{A}_{\mathbb{L}(X, Y)}) \cup  (-\mathcal{A}_{\mathbb{L}(X, Y)})) = S_{\mathbb{L}(X, Y)} $. It is easy to see that $ P $ coincides with the usual operator norm in $ \mathbb{L}(X, Y) $. In addition,  $ \mathcal{M}_T $ coincides with the norm attainment set $ M_T $ of $ T $, i.e., $ \mathcal{M}_T = M_T = \{ x\in S_X : P(T) = p(T(x)) \} $. We should also note that in this setting, $ \mathcal{S}_{ Y^*} = S_{ Y^*} $ and $ \mathcal{S}_{\mathbb{L}(X, Y)^*} = S_{\mathbb{L}(X, Y)^*} $. Therefore, by above theorem, $ T\perp_B A $ if and only if there exist $ x, y \in M_T $ such that $ p(Tx + \lambda Ax) \geq p(Tx) $ for all $ \lambda \geq 0 $ and $ p(Ty + \mu Ay) \geq p(Ty) $ for all $ \mu \leq 0 $. In other words, for $ T, A \in \mathbb{L}(X, Y) $,  $ T\perp_B A $ if and only if there exist $ x, y \in M_T $ such that $ Ax\in Tx^+ $ and $ Ay\in Ty^- $. This completes the proof of the corollary.
\end{proof}

In view of the concept of orthogonality introduced by us in the present article, it is perhaps appropriate to end it with the following remark.\\

\begin{remark}
Theorem $ 1.1 $ of \cite{BS}, also known as the Bhatia-$\breve{S}$emrl Theorem, gives a complete characterization of Birkhoff-James orthogonality of linear operators on finite-dimensional Hilbert spaces. Theorem $ 2.2  $ of \cite{S} generalizes the Bhatia-$\breve{S}$emrl Theorem to linear operators between Banach spaces. Indeed, applying Theorem $ 2.1 $ and Theorem $ 2.2 $ of \cite{SP}, it is easy to see that the Bhatia-$\breve{S}$emrl Theorem follows from Theorem $ 2.2  $ of \cite{S}. On the other hand, as illustrated in Corollary \ref{12. operator orthogonality}, Theorem \ref{11.characterization} of the present article generalizes Theorem $ 2.2  $ of \cite{S}. Therefore, as an application of the concept of orthogonality introduced by us, we obtain a topological version of the Bhatia-$\breve{S}$emrl Theorem. The original Bhatia-$\breve{S}$emrl Theorem considers linear operators on a finite-dimensional Hilbert space. We have generalized this to a much broader context. In our setting, orthogonality of linear operators in $ \mathbb{L}(X,Y) $ is characterized, where $ X $ only needs to be a finite-dimensional vector space with a topology such that $ ((\mathcal{A}_X) \cup (-\mathcal{A}_X)) $ is compact and $ Y $ is a vector space topologized by a finite family of non-trivial semi-norms. The usefulness and applicability of the abstract notions developed in this article is illustrated by the fact that the fundamental principle behind the Bhatia-$\breve{S}$emrl Theorem (and its generalization to Banach spaces, as given in Theorem $ 2.2 $ of \cite{S}) can be immediately generalized to a much broader setting by using these notions.
\end{remark}

\end{document}